\documentclass[12pt]{amsart}

\allowdisplaybreaks[1]
\usepackage{enumerate}
\usepackage{graphicx}
\usepackage{amsmath}
\usepackage{amssymb}
\usepackage{amsfonts}
\usepackage{verbatim}
\usepackage[dvipsnames]{xcolor}
\usepackage{hyperref}
\allowdisplaybreaks
\makeindex

 \newtheorem{theorem}{Theorem}[section]
 \newtheorem{corollary}[theorem]{Corollary}
 \newtheorem{lemma}[theorem]{Lemma}
 \newtheorem{proposition}[theorem]{Proposition}

\newtheorem{observation}[theorem]{Observation}
\theoremstyle{definition}

\theoremstyle{remark}

\newtheorem{fact*}{Fact}

\newcommand\dd{\mathrm d}

\newcommand{\cc}[1]{\overline{#1}}

\newcommand{\tr}{\operatorname{tr}}

\newcommand{\til}{\raise.17ex\hbox{$\scriptstyle\mathtt{\sim}$}}

\newcommand\beq{\begin{equation}}

\newcommand\eeq{\end{equation}}

\newcommand\bbm{\begin{bmatrix}}
\newcommand\ebm{\end{bmatrix}}
\newcommand{\bpm}{ \begin{pmatrix}}
\newcommand{\epm}{\end{pmatrix} }
\numberwithin{equation}{section}

\newlength{\Mheight}
\newlength{\cwidth}

\newcommand{\dfn}[1]{{\bf #1}\index{#1}}

\title[Trace minmax]{Trace minmax functions and the radical Laguerre-P\'olya class}
\author[J. E. Pascoe]{
J. E. Pascoe
}
\address{Department of Mathematics\\
1400 Stadium Rd\\
  University of Florida\\
 Gainesville, FL 32611}
\email[J. E. Pascoe]{pascoej@ufl.edu}

\thanks{J. E. Pascoe is supported by NSF Analysis Grant DMS-1953963}
\date{\today}


\begin{document}

\begin{abstract}
We classify functions $f:(a,b)\rightarrow \mathbb{R}$ which satisfy the inequality $$\tr f(A)+f(C)\geq  \tr f(B)+f(D)$$ when $A\leq B\leq C$ are self-adjoint matrices, $D= A+C-B$,
the so-called \emph{trace minmax functions.} (Here $A\leq B$ if $B-A$ is positive semidefinite, and $f$ is evaluated via the functional calculus.)
A function is trace minmax if and only if its derivative analytically continues to a self map of the upper half plane.
The negative exponential of a trace minmax function $g=e^{-f}$ satisfies the inequality 
	$$\det g(A) \det g(C)\leq  \det g(B) \det g(D)$$
for $A, B, C, D$ as above. We call such functions \emph{determinant isoperimetric}.
We show that determinant isoperimetric functions are in the ``radical" of the the Laguerre-P\'olya class.
We derive an integral representation for such functions which is essentially a continuous version of the Hadamard factorization for functions in the the Laguerre-P\'olya class.
We apply our results to give some equivalent formulations of the Riemann hypothesis.
\end{abstract}

\maketitle

\section{Introduction}
	Let $E\subseteq \mathbb{R}.$
	Let $f:E\rightarrow \mathbb{R}$ be a function.
	Let $X$ be a self-adjoint matrix of size $n$ with spectrum in $E.$
	We now briefly recall how to define $f(X)$ via the \dfn{matrix functional calculus}.
	Let $X$ be diagonalized a unitary matrix $U.$
	That is,
		$$X=U^*\bpm \lambda_1 & & \\ &\ddots & \\ & & \lambda_n \epm U.$$
	We define
		$$f(X)=U^*\bpm f(\lambda_1) & & \\ &\ddots & \\ & & f(\lambda_n) \epm U.$$

	Therefore, for each $n\in\mathbb{N}$, the function $f$ induces a function on $n$ by $n$ self-adjoint matrices with spectrum in $E.$
	Moreover, one can formulate familiar function theoretic notions, such as convexity and monotonicity, in this context.

	Given two self-adjoint matrices $A$ and $B$ we say $A \leq B$ if $B-A$ is positive semidefinite. (This is sometimes called the \dfn{L\"owner order}.)

	Say a function is \dfn{trace monotone} if $A\leq B$ implies $\tr f(A)\leq \tr f(B).$ If we list the eigenvalues of $A$
	as 
		$$\mu_1 \leq \mu_2 \leq \ldots \leq \mu_n,$$
	and those for $B$ as
		$$\lambda_1 \leq \lambda_2 \leq \ldots \leq \lambda_n,$$
	one can show, for example using the Weyl inequalities\cite{WeylIneq}, that $\mu_i\leq\lambda_i.$
	Noting that $\tr f(A) = \sum f(\mu_i)$ and $\tr f(B) = \sum f(\lambda_i),$ we see that \emph{$f$ is trace monotone if and only if $f$ is monotone.}

	Similarly, we say a function is \dfn{trace convex} if $\tr f\left(\frac{A+B}{2}\right)\leq \tr \frac{f(A)+f(B)}{2}.$
	As happened in the case of monotonicity, \emph{a function $f$ is trace convex if and only if $f$ is convex} \cite{gui09, klep2016}.
	In multivariable settings, the theory of joint trace convexity depends intensely on the expression being analyzed \cite{hansen,carlenbook,carlen1}.

	Say a function is \dfn{matrix monotone} if $A\leq B$ implies $f(A)\leq f(B).$
	Let $\Pi$ denote the upper half plane in $\mathbb{C}.$
	L\"owner's theorem states \cite{lo34, bha97} that a function $f:(a,b) \rightarrow \mathbb{R}$ is matrix monotone if and only if $f$ analytically continues to $\Pi$
	and $f: \Pi\cup(a,b) \rightarrow \cc{\Pi}.$
	The Nevanlinna representation \cite{nev22,lax02} then says that
 		$$f(z)= c+dz+\int \frac{1+tz}{t-z} \dd \mu(t)$$
	for some $c \in \mathbb{R}, d\in \mathbb{R}^+$ and positive Borel measure $\mu$ with support contained in $\mathbb{R}\setminus(a,b).$

	Say a function is \dfn{matrix convex} if  $f\left(\frac{A+B}{2}\right)\leq \frac{f(A)+f(B)}{2}.$
	The Kraus theorem states \cite{kraus36, bha97} that a matrix convex function $f:(a,b) \rightarrow \mathbb{R}$ analytically continues to the upper half plane
	and possesses an integral representations similar to, but not the same as, the Nevanlinna representation.

	In general, the current theory of tracial inequalities is real analytic and the theory of matrix inequalities is complex analytic.
	We give a class of trace functions that have nice complex analytic properties, which contrasts to existing literature \cite{hansen,gui09,carlenbook,carlen1,klep2016}.
	
	\subsection{Trace minmax functions}

	Say a function $f$ is \dfn{trace minmax} if $$\tr f(A)+f(C)\geq  \tr f(B)+f(D)$$ whenever $A\leq B\leq C$ are like-sized matrices with spectrum in the domain of $f$
	and $D=A+C-B.$ We use the term ``minmax" because when $A \leq C,$ we can increase $\tr f(A)+f(C)$ by increasing their difference.
	\begin{theorem}\label{mainresult}
	Let $f: (a,b)\rightarrow \mathbb{R}.$
	The following are equivalent:
		\begin{enumerate}
			\item $f$ is trace minmax,
			\item $f'$ is matrix monotone on $(a,b)$,
			\item $f$ analytically continues to the upper half plane $\Pi$ and $f':\Pi\cup(a,b) \rightarrow \cc{\Pi}.$
			\item
			For each $c \in (a,b),$ there exist unique $\alpha, \beta \in \mathbb{R}$
			and a unique finite measure $\mu$ on $[\frac{1}{a-c},\frac{1}{b-c}]$ such that
				$$f(z)=\alpha+\beta z+\int_{[\frac{1}{a-c},\frac{1}{b-c}]} \frac{-\log(1-t(z-c))-t(z-c)}{t^2} \dd \mu.$$
			Here we interpret
				$\frac{-\log(1-t(z-c))-t(z-c)}{t^2}|_{t=0} = z^2.$
				
		\end{enumerate}
	\end{theorem}
	Theorem \ref{mainresult} is proven in Section \ref{MRP}.
	
	Somewhat surprisingly, trace minmax functions are also matrix convex, for the sole reason that $\log x$ is matrix concave on $(0,\infty)$ \cite{bha97}.
	\begin{corollary}\label{matcocor}
	If $f: (a,b)\rightarrow \mathbb{R}$ is trace minmax, then $f$ is matrix convex.
	\end{corollary}
	
	\subsection{The radical Laguerre-P\'olya class}
		We say $f:(a,b)\rightarrow \mathbb{R}^{\geq 0}$ is \dfn{determinant isoperimetric} whenever
			$$\det f(A) \det f(C) \leq \det f(B) \det f(D)$$
		for $A \leq B \leq C$ with spectrum in $(a,b)$ and $D=A+C-B.$
		We use the term ``isoperimetric" because when $A\leq C,$ we can increase the quantity $\det f(A) \det f(C)$ by decreasing the difference between $A$ and $C.$
		Note that $f$ is trace minmax if and only if $e^{-f}$ is determinant isoperimetric.
		Theorem \ref{mainresult} implies that the extreme rays of the cone of trace minmax functions on a neighborhood of zero
		are generated by functions of the form $-\log 1-tx,$
		$x^2,$ $\pm x$ and constants. Therefore, $1-tx,$ $e^{-x^2},$ $e^{\pm x}$ and constant functions are determinant isoperimetric.
		Thus, we obtain the following system of inequalities.
		\begin{corollary}\label{lamecor}
		Let $A, B, C \in M_n(\mathbb{C})$ such that $A \leq B \leq C.$ Let $D = A+C-B.$
		The following are true:
			\begin{enumerate}
				\item $\det e^A \det e^C = \det e^B \det e^D,$
				\item $\det e^{B^2} \det e^{D^2} \leq \det e^{A^2} \det e^{C^2},$ and thus, $$\|B\|_F+\|D\|_F\leq \|A\|_F + \|C\|_F,$$
				\item for all $t \in \left(-\frac{1}{\|A\|}, \frac{1}{\|C\|}\right),$
					$$\det 1-tA \det 1-tC \leq \det 1-tB \det 1-tD.$$
			\end{enumerate}
		\end{corollary}
		In principle, these generate (under the operations of products, $n$-th roots, and taking limits) all inequalities of the form 
			$$\prod f(\alpha_i) \prod f(\gamma_i) \leq \prod f(\beta_i) \prod f(\delta_i)$$
		where $\alpha_i, \beta_i, \gamma_i, \delta_i$ are the eigenvalues of $A, B, C, D$ respectively, where $A\leq B\leq C$ and $D=A+C-B.$
		One wonders if there is a classification of all eigenvalue inequalities satisfied by $D$ such that $D = A+C-B$ where $A \leq B \leq C$
		along the lines of Horn's conjecture \cite{HornConjecture} and the Knutson-Tao theorem \cite{knutsontao}. 

		The function $-\log x$ is trace minmax on $(0,\infty),$ and therefore $x$ is determinant isoperimetric there, yielding a
		more memorable inequality along the lines of the characteristic polynomials inequality in item 3 in Corollary \ref{lamecor}. 
		\begin{corollary}[Isoperimetric inequality]
		Let $A, B, C \in M_n(\mathbb{C})$ such that $0\leq A \leq B \leq C.$ Let $D = A+C-B.$
		Then,
			$$\det A \det C \leq \det B \det D.$$
		\end{corollary}

		The set of determinant isoperimetric functions is closed under multiplication and pointwise convergent limits.
		Moreover, as $1-tx,$ $e^{-x^2},$ $e^{\pm x}$ and constant functions are determinant isoperimetric, we see that
		any Hadamard product of the form
			\beq \label{hadamard} f(x) = x^ke^{-a-bx-cx^2}\prod (1-x/\rho_i)e^{x/\rho_i} \eeq
		where $b\in \mathbb{R},$ $c\geq 0$ is determinant isoperimetric on open intervals in $\mathbb{R}$ where $f$ takes nonnegative values.
		The \dfn{Laguerre-P\'olya class} is the set of entire functions which are the locally uniform limits of real-rooted polynomials.
		Laguerre-P\'olya class functions are important in various contexts, \cite{debrangesentire, debrangesjfa, szasz1943, polya,biehler}. 
		The Laguerre-P\'olya class is exactly the set of functions of the form \eqref{hadamard}.
		Define the \dfn{radical Laguerre-P\'olya class of $(a,b)$} to be the set of functions on $(a,b)$ which are the pointwise limits of real $n$-th roots
		of functions in the Laguerre-P\'olya class which are on nonnegative $(a,b).$

		Evidently, negative exponentials of trace minmax functions are exactly the radical Laguerre-P\'olya class of $(a,b)$
		as the cone of of trace minmax functions is generated by by functions of the form $-\log 1-t(x-c),$
		$x^2,$ $\pm x$ and constants and their negative exponentials are in the Laguerre-P\'olya class
		\begin{theorem}\label{LP}
			Let $f: (a,b)\rightarrow \mathbb{R}.$
			The following are equivalent:
			\begin{enumerate}
				\item $f$ is trace minmax,
				\item $e^{-f}$ is determinant isoperimetric,
				\item $e^{-f}$ is in the radical Laguerre-P\'olya class.
			\end{enumerate}
		\end{theorem}

\section{Preliminaries}
	\subsection{Derivatives in the functional calculus}
	We adopt the following notation for derivatives taken in the functional calculus,
		$$Df(X)[H] = \lim_{t\rightarrow 0} \frac{f(X+tH)-f(X)}{t},$$  $$D^2f(X)[H,K] = \lim_{t} \frac{Df(X+tK)[H]-Df(X)[H]}{t},$$
	where $X, H, K$ are like-sized self-adjoint matrices.
	
	\begin{lemma}\label{deriveq}
		If $f$ is analytic, trace minmaxity is equivalent to saying that $D^2f(X)[H,K] \geq 0$ whenever $H,K\geq 0.$
	\end{lemma}
	\begin{proof}
		First, suppose $f$ is trace minmax.
		Let $X$ be a self-adjoint matrix and let $H, K \geq 0.$
		Note $X \leq X+tH \leq X+tH+sK.$
		So,
			$f(X+tH+sK)+f(X)\leq f(X+sK)+f(X+tH).$
		Therefore,
			$$\frac{f(X+tH+sK)+f(X)- f(X+sK)-f(X+tH)}{st} \geq 0.$$
		Taking the limit as $t\rightarrow 0,$ we see that
			$$\frac{Df(X+sK)[H]- Df(X)[H]}{s} \geq 0.$$
		Now taking $s\rightarrow 0,$ $D^2f(X)[H,K]\geq 0.$
		
		To see the converse, let $A \leq B \leq C.$ Let $H = B-A, K=C-B.$
		Now, $Df(A+tH+sK)[H,K] \geq 0.$
		Next,
		\begin{align*}
			0&\leq \int^1_{0} Df(A+tH+sK)[H,K] \dd t\\ &=Df(B+sK)[K]-Df(A+sK)[K].
		\end{align*}
		Finally,
		\begin{align*}
		0 &\leq \int^1_{0}Df(B+sK)[K]-Df(A+sK)[K] \dd s\\ &= f(A)+f(C)-f(B)-f(A+C-B).\end{align*}
	\end{proof}

	\subsection{Nevanlinna's solutions to moment problems}
		In 1922, Nevanlinna considered the question of when a sequence $\rho_n$ is a sequence of moments for some finite positive Borel measure.
		The problem is intimately connected to the theory of self maps of the upper half plane.
		\begin{theorem}[\cite{nev22}]
			Let $\rho_n$ be a sequence of real numbers. Let $a, b>0$
			The following are equivalent:
			\begin{enumerate}\label{nevanlinna1}
				\item There exists a positive Borel measure $\mu$ on $[\frac{-1}{a},\frac{1}{b}]$ such that $\rho_n = \int t^n \dd \mu,$
				\item The moment generating function $f(z) = \sum^{\infty}_{n=0} a_nz^{n+1}$ analytically continues to $\Pi\cup (a,b)$
				and $f:\Pi \cup (a,b) \rightarrow \overline{\Pi}.$
			\end{enumerate}
		\end{theorem}
		There is also a nice Hankel matrix type condition. (In fact, this is used in conjunction with a GNS-type construction to prove the prior theorem.)
		\begin{theorem}[\cite{nev22}] \label{nevanlinna2}
			Let $\rho_n$ be a sequence of real numbers. 
			The following are equivalent:
			\begin{enumerate}
				\item There exists a positive Borel measure $\mu$ on $\mathbb{R}$ such that $\rho_n = \int t^n \dd \mu,$
				\item The infinite Hankel matrix $$\bbm
				\rho_0 & \rho_1 & \rho_2 & \ldots\\
				\rho_1 & \rho_2 & \rho_3 & \ldots\\
				\rho_2 & \rho_3 & \rho_4 & \ldots\\
				\vdots & \vdots & \vdots & \ddots
				\ebm$$
				is positive semidefinite.
			\end{enumerate}
		\end{theorem}

\section{Trace duality}
	We now endeavor to show that
		$$\tr Df(X)[H]=\tr Hf'(X),$$
	which we will use later.

	For example, consider $f(x) = x^3.$
	The derivative is given by
		$$Df(X)[H] = HX^2+ XHX + X^2H.$$
	Note,
		$$\tr Df(X)[H] = \tr H 3X^2 = \tr Hf'(X).$$
	It is clear that an inductive argument would prove this for polynomials. However, for general functions, matters are a bit more delicate.
	Our approach uses algebraic manipulation in the functional calculus. It is also likely there is a somewhat involved argument using Stone-Weierstrauss.
	\begin{lemma} \label{unitaryinvar}
	Let $f:(a,b)\rightarrow \mathbb{R}$ be a function.
	Let $U$ be a unitary. Then,
		 $$f(U^*XU)=U^*f(X)U.$$
	\end{lemma}
	\begin{proof}
		Suppose the unitary $V$ diagonalizes $X.$
		$$f(X)=V^*\bpm f(\lambda_1) & & \\ &\ddots & \\ & & f(\lambda_n) \epm V.$$
		Now, $VU$ diagonalizes $U^*XU,$ and so
		\begin{align*}f(U^*XU)&=U^*V\bpm f(\lambda_1) & & \\ &\ddots & \\ & & f(\lambda_n) \epm VU\\
		&=U^*f(X)U\end{align*}
	\end{proof}

	\begin{lemma} \label{unitaryconj}
		Let $f:(a,b)\rightarrow \mathbb{R}$ be a function.
		Let $U$ be a unitary. Then,
			 $$Df(U^*XU)[U^*HU]=U^*Df(X)[H]U.$$
	\end{lemma}
	\begin{proof}
		Calculating using Lemma \ref{unitaryinvar}
			\begin{align*}				
			Df(U^*XU)[U^*HU] &= \lim_{t\rightarrow 0} \frac{f(U^*XU+tU^*HU)-f(U^*XU)}{t}\\
			&= \lim_{t\rightarrow 0} \frac{f(U^*(X+tH)U)-f(U^*XU)}{t}\\
			&= \lim_{t\rightarrow 0} \frac{U^*f(X+tH)U-U^*f(X)U}{t}\\
			&= \lim_{t\rightarrow 0} \frac{U^*(f(X+tH)-f(X))U}{t}\\
			&=U^*\left(\lim_{t\rightarrow 0} \frac{f(X+tH)-f(X)}{t}\right) U\\
			&= U^*Df(X)[H]U.
			\end{align*}
	\end{proof}

	\begin{theorem} \label{duegar}
		Let $f:(a,b)\rightarrow \mathbb{R}$ be a $C^1$ function.
		Then,
			 $$\tr Df(X)[H]=\tr Hf'(X).$$
	\end{theorem}
	\begin{proof}
		Because $f$ is smooth, for each self-adjoint matrix $X$ with spectrum in $(a,b)$,
		$\tr Df(X)[H]$ is linear map from $n \times n$ matrices to $n \times n$ matrices as a function of $H$
		and there is a unique quantity $g(X)$ such that $\tr f(X)[H] = \tr Hg(X).$
		We will show that:
		\begin{enumerate}
			\item $g(U^*XU)=U^*g(X)U$ for all unitaries $U,$
			\item $g(X_1\oplus X_2) = g(X_1)\oplus g(X_2),$
			\item $g(x) = f'(x)$ whenever $x$ is a real number in $(a,b)$.
		\end{enumerate}

		To see (1), note that by Lemma \ref{unitaryconj} $$Df(U^*XU)[H]=U^*Df(X)[UHU^*]U.$$
		Therefore,
			\begin{align*}				
			\tr Hg(U^*XU)	&=\tr Df(U^*XU)[H]\\
					&=\tr U^*Df(X)[UHU^*]U\\ 
					&=\tr Df(X)[UHU^*]\\ 
					&= \tr UHU^*g(X)\\
					&= \tr HU^*g(X)U\\
			\end{align*}
		So, $g(U^*XU)= U^*g(X)U.$
		
		To see (2), first 
		write
			$$H = \bbm H_{11} & H_{12} \\ H_{21} & H_{22} \ebm.$$
		Note that $f(X_1\oplus X_2) = f(X_1)\oplus f(X_2),$ therefore
			$Df(X_1\oplus X_2)[H_{11}\oplus H_{22}]= Df(X_1)[H_{11}] \oplus Df(X_2)[H_{22}].$
		Translating the relation to $g,$ one sees that $g(X_1\oplus X_2)$ is of the form:
			$$g\bpm  X_1 & \\ & X_2\epm=\bbm g(X_1) & A(X_1,X_2) \\ A(X_2,X_1) & g(X_2) \ebm$$
		for some unknown quantities $A(X_1,X_2), A(X_2,X_1).$
		Now by (1),
			\begin{align*}
				g\bpm  X_1 & \\ & X_2\epm&=g\left(\bpm 1 & \\ & -1\epm\bpm  X_1 & \\ & X_2\epm\bpm 1 & \\ & -1\epm\right)\\
							&=\bpm 1 & \\ & -1\epm g\bpm  X_1 & \\ & X_2\epm \bpm 1 & \\ & -1\epm\\
							&=\bbm g(X_1) & -A(X_1,X_2) \\ -A(X_2,X_1) & g(X_2) \ebm,			
			\end{align*}
		and therefore $A(X_1,X_2), A(X_2,X_1)$ both equal $0.$ Thus, $g(X_1\oplus X_2) = g(X_1)\oplus g(X_2).$
		
		Now to see (3), let $x$ be a real number. Note $$\tr Df(x)[h]=Df(x)[h] = hf'(x) = \tr hf'(x),$$ and therefore $g(x) = f'(x).$
		
		We now claim $f'(X) = g(X).$
		Write 
			$$X=U^*\bpm \lambda_1 & & \\ &\ddots & \\ & & \lambda_n \epm U.$$
		Now,
		\begin{align*}
			f'(X)&= U^*\bpm f'(\lambda_1) & & \\ &\ddots & \\ & & f'(\lambda_n) \epm U\\
			&= U^*\bpm g(\lambda_1) & & \\ &\ddots & \\ & & g(\lambda_n) \epm U\\
			&= U^*g\bpm \lambda_1 & & \\ &\ddots & \\ & & \lambda_n \epm U\\
			&= g\left(U^*\bpm \lambda_1 & & \\ &\ddots & \\ & & \lambda_n \epm U\right)\\
			&= g\left(X \right).\\
		\end{align*} 
	\end{proof}

\section{Derivatives of trace minmax functions are matrix monotone}
	\begin{lemma}\label{derivmonoc1}
		Let $f: (a,b)\rightarrow \mathbb{R}$ be $C^1.$
		The function $f$ is trace minmax if and only if $f'$ is matrix monotone on $(a,b).$
	\end{lemma}
	\begin{proof}
		Let $A \leq B \leq C.$
		One can rewrite the defining inequality for trace minmaxity
			$$\tr f(A)+f(C)\geq \tr f(B)+f(A+C-B)$$
		as
			$$\tr f(C)-f(B)\geq \tr f(A+C-B)-f(A)$$
		Let $C = B+tH.$
		Now
			$$\tr f(B+tH)-f(B)\geq \tr f(A+tH)-f(A).$$
		Dividing by $t$ and taking the limit as $t \rightarrow 0$ gives
			$$\tr Df(B)[H]\geq \tr Df(A)[H].$$
		Applying trace duality established in Theorem \ref{duegar}, we see that
		$$\tr Hf'(B)\geq \tr Hf'(A).$$
		Now, $\tr H(f'(B)-f'(A)) \geq 0$ for an arbitrary positive semidefinite matrix $H$ and therefore $f'(B)-f'(A)$ is positive semidefinite.
		Therefore $f'(A)\leq f'(B)$ and so $f'$ is matrix monotone.
	\end{proof}

	\begin{theorem} \label{posmono}
		Let $f: (a,b)\rightarrow \mathbb{R}.$
		The function $f$ is trace minmax if and only if $f'$ is matrix monotone on $(a,b).$
	\end{theorem}
	\begin{proof}
		Without loss of generality $a=-1$ and $b=1$
		First observe that as a function on $(-1,1),$ $f$ is convex, and therefore continuous.
		Fix $\varphi$ a positive smooth function such that $\int_{\mathbb{R}} \varphi = 1$ with support contained in $(-1,1).$
		Write $\varphi_t(x)= \varphi(x/t)/t.$
		Write $f_t = f * \varphi_t.$
		Note $f_t$ is trace minmax on $(-1+t,1-t).$ Therefore, by Lemma \ref{derivmonoc1}, $f_t'$ is matrix monotone on $(-1+t,1-t).$
		As $f_t \rightarrow f$ as $t\rightarrow 0$ because $f$ is continuous, and a pointwise limit of matrix monotone functions is matrix monotone, we are done.

		To see the converse,
		note that, if $f'$ is matrix monotone and $H,K$ are positive semidefinite,
		$$\tr D^2f(X)[H,K] = \tr HDf'(X)[K]\geq 0,$$
		so we are done by Lemma \ref{deriveq}.

	\end{proof}

\section{Trace minmax representation theorems}
	We now prove our representation theorem for trace minmax functions.
	\begin{proposition} \label{repthm}
		Let $f:(a,b)\rightarrow \mathbb{R}.$
		If $f$ is trace minmax then for each $c \in (a,b),$ there exists a unique measure $\alpha, \beta \in \mathbb{R}$
			and a unique finite measure $\mu$ on $[\frac{1}{a-c},\frac{1}{b-c}]$ such that
				$$f(z)=\alpha+\beta z+\int_{[\frac{1}{a-c},\frac{1}{b-c}]} \frac{-\log(1-t(z-c))-t(z-c)}{t^2} \dd \mu.$$
	\end{proposition}
	\begin{proof}
		Without loss of generality $c =0.$		
		Because $f$ is trace minmax, by Theorem \ref{posmono}, $f'$ is matrix monotone.
		Furthermore, by L\"owner's theorem, $f$ analytically continues to an analytic function $f:(a,b)\cup\Pi \rightarrow \cc\Pi.$
		Write $f(z) = a_nz^n.$ As $f'(z)$ is self map of the upper half plane,
		there is a measure $\mu$ supported on $[\frac{1}{a},\frac{1}{b}]$ such that $na_n = \int t^{n-2}\dd \mu$
		by Nevanlinna's solution to the Hamburger moment problem \cite{nev22}, which we gave as Theorem \ref{nevanlinna1}.
		Now,
		\begin{align*}		
		f(z) & = a_0+ a_1z+ \sum^\infty_{n=2} \frac{z^n\int t^{n-2} \dd \mu}{n} \\
		& = a_0+ a_1z+ z^2\sum^\infty_{n=0} \frac{z^n\int t^{n}\dd \mu}{n+2} \\
		& = a_0+ a_1z+ z^2\sum^\infty_{n=0}  \int\frac{ (zt)^{n}}{n+2}\dd \mu \\
		& = a_0+a_1z+z^2\int \frac{-\log(1-tz)-tz}{(zt)^2} \dd \mu\\
		&= a_0+a_1z+\int \frac{-\log(1-tz)-tz}{t^2} \dd \mu.
		\end{align*}
	\end{proof}

	A consequence of the fact that $f'(z)$ is a Pick function and Theorem \ref{nevanlinna2} is a Hankel matrix type test for trace minmaxity.
	\begin{observation} \label{hankelobs}
		Let $f(x) = \sum a_n x^n$ be a convergent series on a neighborhood of $0.$
		The function $f$ is trace minmax if and only if the Hankel matrix
			$$\bbm
				2a_2 & 3a_3 & 4a_4 & \ldots\\
				3a_3 & 4a_4 & 5a_5 & \ldots\\
				4a_4 & 5a_5 & 6a_6 & \ldots\\
				\vdots & \vdots & \vdots & \ddots
			\ebm$$
		is positive semidefinite.		
	\end{observation}

\section{Proof of the main result} \label{MRP}
	(1) $\Leftrightarrow$ (2) is Theorem \ref{posmono}. (2) $\Leftrightarrow$ (3) is L\"owner's theorem. (1) $\Rightarrow$ (4)
	is Proposition \ref{repthm}. (4) $\Rightarrow$ (3)
        The derivative of such an integral representation is 
			$$b+ \int_{[\frac{1}{a-c},\frac{1}{b-c}]}\frac{z}{1-tz} \dd \mu.$$
		Since each $\frac{z}{1-tz}$ takes the upper half plane to itself, so does whole formula.

\section{Examples}
	We now give some examples.
	\begin{enumerate}
		\item The function $e^z,$ real-rooted polynomials, and the Gamma function are all determinant isoperimetric by virtue of being in the
		Laguerre-P\'olya class.
		\item The function $x^t$ for $1\leq t\leq 2$ is trace minmax, because the derivative is a self-map of the upper half plane.
		\item Consider Riemann's original $\Xi$ function.
		That is, take
			$$\xi(z) = \frac{1}{2}z(z-1)\pi^{s/2}\Gamma(z/2)\zeta(z),$$
		and define
			$\Xi(z) = \xi(1/2+iz).$
		The Riemann hypothesis says that the zeros of $\Xi$ are real. Moreover, we know $\Xi(z) = \prod (1-\frac{z}{\rho_i})e^{z/\rho_i}$
		are $1/2+i\rho_i$ are the nontrivial zeros of the Riemann zeta function. Therefore, if 
		the Riemann hypothesis is true, then $\Xi$ is in the Laguerre-P\'olya class. 
		 Applying our results in tandem, we see the
		following list of equivalent statements to the Riemann hypothesis.
		\begin{proposition}
			Let $(a,b)$ be a nonempty open interval in $\mathbb{R}$ where $\Xi$ is nonvanishing.
			The following are equivalent:
			\begin{enumerate}
				\item the Riemann hypothesis is true,
				\item $\Xi$ is in the (radical) Laguerre-P\'olya class of $(a,b)$,
				\item $\log \Xi (z)$ has a branch defined on the upper half plane,
				\item $|\Xi|$ is determinant isoperimetric on $(a,b),$
				\item $-\log |\Xi (z)|$ is trace minmax on $(a,b)$,
				\item $-\log |\Xi (z)|$ is matrix convex on $(a,b)$,
				\item $-\frac{\dd}{\dd z} \log|\Xi(z)|$ is matrix monotone $(a,b)$,
				\item Let $r \in (a,b).$ If we write $-\log \Xi(z+r)= \sum a_nz^n,$ then the infinite matrix,
					$$\bbm
					2a_2 & 3a_3 & 4a_4 & \ldots\\
					3a_3 & 4a_4 & 5a_5 & \ldots\\
					4a_4 & 5a_5 & 6a_6 & \ldots\\
					\vdots & \vdots & \vdots & \ddots
					\ebm,$$
				is positive semidefinite.
				\item Let $r \in (a,b).$ If we write $-\log \Xi(z+r)= \sum a_nz^n,$ then the infinite matrix,
					$$\bbm
					a_2 & a_3 & a_4 & \ldots\\
					a_3 & a_4 & a_5 & \ldots\\
					a_4 & a_5 & a_6 & \ldots\\
					\vdots & \vdots & \vdots & \ddots
					\ebm,$$
				is positive semidefinite.
				\item Let $r \in (a,b).$ If we write $-\log \Xi(z+r)= \sum a_nz^n,$ then there exists a $k \in \mathbb{N}$ such that the infinite matrix,
					$$\bbm
					2ka_{2k} & (2k+1)a_{2k+1} & (2k+2)a_{2k+2} & \ldots\\
					(2k+1)a_{2k+1} & (2k+2)a_{2k+2} & (2k+3)a_{2k+3} & \ldots\\
					(2k+2)a_{2k+2} & (2k+3)a_{2k+3} & (2k+4)a_{2k+4} & \ldots\\
					\vdots & \vdots & \vdots & \ddots
					\ebm,$$
				is positive semidefinite.
				\item Let $r \in (a,b).$ If we write $-\log \Xi(z+r)= \sum a_nz^n,$ then there exists a $k \in \mathbb{N}$ such that the infinite matrix,
					$$\bbm
					a_{2k} & a_{2k+1} & a_{2k+2} & \ldots\\
					a_{2k+1} & a_{2k+2} & a_{2k+3} & \ldots\\
					a_{2k+2} & a_{2k+3} & a_{2k+4} & \ldots\\
					\vdots & \vdots & \vdots & \ddots
					\ebm,$$
				is positive semidefinite.
			\end{enumerate}
		\end{proposition}
		\begin{proof}
			$(a) \Leftrightarrow (b)$ is classical \cite{polya2} and follows directly from the Hadamard factorization of $\Xi$.

			$(a) \Leftrightarrow (c)$ $\Xi$ is nonvanishing on the upper half plane if and only if it admits a branch of the logarithm.

			$(b) \Leftrightarrow (d) \Leftrightarrow (e)$ is Theorem \ref{LP}.

			$(e) \Leftrightarrow (g)$ is part of Theorem \ref{mainresult}.

			$(e) \Rightarrow (f)$ is Corollary \ref{matcocor}.

			$(f) \Rightarrow (c)$ is Kraus theorem \cite{kraus36}.

			$(e) \Leftrightarrow (h)$ follows from Observation \ref{hankelobs}.

			$(h) \Rightarrow (i)$
			Note that
				$$\bbm
					1/2 & 1/3 & 1/4 & \ldots\\
					1/3 & 1/4 & 1/5 & \ldots\\
					1/4 & 1/5 & 1/6 & \ldots\\
					\vdots & \vdots & \vdots & \ddots
				\ebm \geq 0,$$
			and, therefore,
				$$\bbm
					1/2 & 1/3 & 1/4 & \ldots\\
					1/3 & 1/4 & 1/5 & \ldots\\
					1/4 & 1/5 & 1/6 & \ldots\\
					\vdots & \vdots & \vdots & \ddots
				\ebm \cdot \bbm
					2a_2 & 3a_3 & 4a_4 & \ldots\\
					3a_3 & 4a_4 & 5a_5 & \ldots\\
					4a_4 & 5a_5 & 6a_6 & \ldots\\
					\vdots & \vdots & \vdots & \ddots
				\ebm = 
				\bbm
					a_2 & a_3 & a_4 & \ldots\\
					a_3 & a_4 & a_5 & \ldots\\
					a_4 & a_5 & a_6 & \ldots\\
					\vdots & \vdots & \vdots & \ddots
				\ebm \geq 0.$$
			
			$(h)\Rightarrow (j)$ is trivial.

			$(j)\Rightarrow (k)$ has essentially the same proof as $(h) \Rightarrow (i).$

			$(i)\Rightarrow (k)$ is trivial.

			$(k) \Rightarrow (c)$ follows from Theorem \ref{nevanlinna2} combined with \ref{nevanlinna1} applied to the function $\sum^{\infty}_{j=1} a_{2k+j}z^{2k+j+1}.$
		\end{proof}
		The above formulation of the the Riemann hypothesis evokes a similiarity to approaches using hyperbolicity of Jensen polynomials
		taken in \cite{polya2, ono1, ono2}, and a positivity of derivatives approach in Li's criterion \cite{Li}.
	\end{enumerate}


\bibliography{references}
\bibliographystyle{plain}

\end{document}